\documentclass[reqno,11pt]{amsart}

\usepackage{amssymb,graphics,graphicx,wrapfig,amssymb,epsfig,graphicx,epsf,amsmath,amsthm,subfig}

\def\eps{\varepsilon}

\def\be{\begin{equation}}
\def\ee{\end{equation}}
\def\ba{\begin{align}}
\def\bm{\begin{multline}}
\def\bfig{\begin{figure}[htb]}
\def\efig{\end{figure}}

\numberwithin{equation}{section}
\newtheorem{theorem}{Theorem}[section]
\newtheorem{proposition}[theorem]{Proposition}
\newtheorem{lemma}[theorem]{Lemma}

\newtheorem{remark}{Remark}

\newcommand{\R}{\mathbb{R}}
\newcommand{\si}{\sigma}

\newcommand{\Prob}{\mathbb{P}}
\newcommand{\ve}{\varepsilon}

\newcommand{\Var}{\text{Var}}
\newcommand{\deter}{\text{det}}

\newcommand{\tU}{\tilde{U}}
\newcommand{\tH}{\tilde{H}}

\newcommand{\mcO}{\mathcal{O}}

\newcommand{\mcD}{\mathcal{D}}

\newcommand{\sez}{\sigma \downarrow 0}

\DeclareMathSymbol{\leqslant}{\mathalpha}{AMSa}{"36}
\DeclareMathSymbol{\geqslant}{\mathalpha}{AMSa}{"3E}
\DeclareMathSymbol{\doteqdot}{\mathalpha}{AMSa}{"2B}
\DeclareMathSymbol{\circlearrowright}{\mathalpha}{AMSa}{"08}
\DeclareMathSymbol{\subsetneq}{\mathalpha}{AMSb}{"28}
\DeclareMathSymbol{\supsetneq}{\mathalpha}{AMSb}{"29}
\renewcommand{\leq}{\;\leqslant\;}
\renewcommand{\geq}{\;\geqslant\;}
\newcommand{\isdefby}{\; := \;}
\newcommand{\bydefis}{\; =: \;}

\newcommand{\dd}{{\rm d}}
\newcommand{\e}[1]{\,{\rm e}^{#1}\,}

\newcommand{\upchi}{\raise 2pt \hbox{$\chi$}}

\def\writefig#1 #2 #3 {\rlap{\kern #1 truecm \raise #2 truecm
\hbox{#3}}}

\newcommand{\caC}{{\mathcal C}}
\newcommand{\caD}{{\mathcal D}}

\newcommand{\bbR}{{\mathbb R}}

\newcommand{\bsx}{{\boldsymbol x}}

\begin{document}


\title{Breaking the chain}
\author{Michael Allman and Volker Betz}

\address{Michael Allman \hfill\newline
\indent Department of Mathematics \hfill\newline
\indent University of Warwick \hfill\newline
\indent Coventry, CV4 7AL, England \hfill
}
\email{m.j.allman@warwick.ac.uk}

\address{Volker Betz \hfill\newline
\indent Department of Mathematics \hfill\newline
\indent University of Warwick \hfill\newline
\indent Coventry, CV4 7AL, England \hfill\newline
{\small\rm\indent http://www.maths.warwick.ac.uk/$\sim$betz/} 
}
\email{v.m.betz@warwick.ac.uk}

\maketitle

\begin{quote}
{\small
{\bf Abstract.}
We consider the motion of a Brownian particle in $\R$, moving between a particle fixed at the origin and another moving deterministically away at slow speed $\ve>0$.  The middle particle interacts with its neighbours via a potential of finite range $b>0$, with a unique minimum at $a>0$, where $b<2a$.  We say that the chain of particles breaks on the left- or right-hand side when the middle particle is greater than a distance $b$ from its left or right neighbour, respectively.  We study the asymptotic location of the first break of the chain in the limit of small noise, in the case where $\ve = \ve(\sigma)$ and $\si>0$ is the noise intensity.
}  

\vspace{1mm}

\vspace{1mm}
\noindent
{\footnotesize {\it 2000 Math.\ Subj.\ Class.:} 60J70}
\end{quote}


\section{Introduction}

We are interested in the behaviour of a chain of interacting particles while it is pulled beyond its breaking point. Obvious real world examples would include the tearing of a band of rubber, or a rope, and obvious questions would be how much strain a given chain can endure before breaking, and where the breakpoint will be located once it occurs. 

A model for such a chain is given by a collection of interacting Brownian particles i.e. one investigates solutions of the SDE system 
\be \label{long chain}
\dd x_i(t) = - \frac{\partial H}{\partial x_i}(\bsx(t)) \, \dd t + \sigma \dd W_i(t),
\ee
where $(W_i)_{1 \leq i \leq N}$ are independent Brownian motions, $\bsx = (x_1, \ldots x_N) \in \bbR^N$ is the collection of particle positions, and $\sigma$ is the (small) noise intensity. The potential energy of the chain is given by 
\[
H(\bsx) = \sum_{1 \leq i < j \leq N} U(x_i-x_j)
\]
for some pair potential $U$. We now exert strain on this chain of interacting Brownian particles.  This is done by solving (\ref{long chain}) only for $2 \leq i \leq N-1$, fixing $x_1=0$ and pulling $x_N$ outwards with (slow) speed $\eps$; the starting configuration of the chain should be a stable equilibrium, ideally a global minimum configuration of the potential energy. The mathematical questions corresponding to the problems above are then about the expected time of a (still to be defined) breaking event, and its location along the chain. In our case, the potential $U$ will have compact support, on $|x| \leq K$, say, and the breaking event will occur when the distance between two given particles is greater than $K$. 

The model (\ref{long chain}) is widely used in material science to model the dynamics of crystals, in particular the propagation of cracks. Investigations there are purely numerical, and the main difficulty is the sheer size of the system under consideration. Out of the vast literature on the topic, we only mention \cite{EMcC78,BvG82,PE02,WHV92} and the references therein. 

In mathematics, a model of type (\ref{long chain}) has recently been investigated by 
T.\ Funaki \cite{fun1, fun2}. He studies the free motion of a, possibly multi-dimensional, crystal of interacting Brownian particles. In the limit of zero temperature, and under suitable assumptions on $U$, he shows  that if the system is initially rigidly crystallized, then it stays so for macroscopic time, and that the crystal as a whole performs Brownian motion both in the translational and in the rotational degrees of freedom. This is, in some sense, the opposite situation to when the crystal is torn apart by force. 

It is clear that in the situation of stretching the chain (\ref{long chain}) until two particles are more than $K$ apart, we are looking at a first exit problem from a time-dependent domain. Also, although the chain is one-dimensional, the first-exit problem is not, indeed it is $(N-2)$-dimensional. 

The problem of first-exit from a stationary potential well has been studied in great detail.  In \cite{fw}, the expected exit time from such a well is shown to behave asympotically like $\e{2h/\sigma^2}$, where $h$ is the height of potential well to overcome.  This is in agreement with the classical Eyring-Kramers formula \cite{ey,kr}.  In the multidimensional case, a proof of the expected exit time, with prefactor, has been given only recently \cite{ab}.

The case of a moving potential well is even more difficult and thus for the time being we settle for a further simplification: we take $N=3$, and $U$ as a cut-off strictly convex potential. In this case, only $x_2$ is moving, and so the problem to solve now is the exit of a one-dimensional stochastic process from a time-dependent domain, which still is a rather difficult and interesting topic,
and is related to the theory of stochastic resonance \cite{nb:pathwise, nb:beyond, nb:sample-paths}. Additionally, while in \cite{nb:pathwise, nb:beyond, nb:sample-paths} usually only the exit time distribution is of interest, we will need to know on which side of the domain the exit occurs. This question cannot, to our knowledge, be answered by any previously available results. So we solve it by direct investigation of the SDE, using some of the theory from \cite{nb, nb:pathwise}.

Our main result is Theorem \ref{main}. Roughly speaking, it says that for pulling speed $\eps$ and noise level $\sigma$ both going to zero, the chain will almost surely break on the right hand side if $\ve > \sigma \sqrt{|\ln \sigma|}$, while it will break on either side with probability $1/2$ when $\ve < \sigma/ \sqrt{|\ln \sigma|}$. This corresponds to the intuition that pulling too fast will just rip off the final 
particle of the chain, as the noise does not have time to bring the configuration back to an equilibrium. Conversely, pulling very slowly corresponds to an adiabatic situation where the chain is in its local energy minimum all the time and the $1/2$ exit probability follows by symmetry. What is surprising is that we obtain this picture with great precision, with both cases separated only by a factor of $|\ln \sigma|$. 
We do not know what happens in between the two cases specified in Theorem \ref{main}, although it is likely that the almost-sure law will start to fail before $\sigma = \eps$ due to the fluctuations of Brownian motion. The asymptotic behaviour of the system at this point or, for that matter, at any constellation with $\sigma \sqrt{|\ln \sigma|} \leq \ve \leq \sigma / \sqrt{|\ln \sigma|}$ is an interesting, but probably rather hard, open problem.
\\[1mm]
{\bf Acknowledgements:} We would like to thank Nils Berglund, Barbara Gentz and Anton Bovier for their valuable comments and stimulating discussions.

\section{The model and main result}
Three particles $x_L$, $x$ and $x_R$ in $\R$ interact with each other via a potential, $U$, of finite range satisfying
\begin{itemize}
	\item[(U0)] $U \in \caC(\R)$ with $U(-y)\isdefby U(y)$
	\item[(U1)] $U(y) = 0$ for $|y| \geq b$ and $U \in \caC^3((0,b))$
	\item[(U2)] $b < 2a$, where $U(a) = \min_{y \geq 0} U(y)$
	\item[(U3)] There exists $a_0 \in (0,a)$ such that $U''(y) \geq u_0 > 0$ for all $y \in (a_0,b)$
\end{itemize}
The particle $x_L$ is fixed at the origin and the position of $x_R$ at time $s$ is given by $x_R(s)=2a(1+\ve s)$, where $\ve > 0$ is a small parameter.  We study the behaviour of the middle particle, with position at time $s$ given by $x_s$.  Initially, it has position $x_0 = a$ so that the distance between neighbouring particles is $a$, which is the energetically-optimal configuration for this potential.  The time-dependent potential energy of the particle at position $x$ is given by
\[
H(x,\ve s) = U(x) + U(2a(1+\ve s)-x)
\]
The middle particle $x$ moves according to the SDE
\be\label{eq:sdes}
\dd x_s = -\frac{\partial H}{\partial x}(x_s,\ve s) \, \dd s + \sigma \dd W_s
\ee
where $x_0 = a$, $W_s$ is a standard Brownian motion and $\sigma>0$ is the noise intensity.  Rescaling time as $t=\ve s$, this is the same in distribution as solving
\begin{align}
\dd x_t & = -\frac{1}{\ve}\frac{\partial H}{\partial x}(x_t,t) \, \dd t + \frac{\sigma}{\sqrt{\ve}} \dd W_t \label{eq:xt}\\
& = \frac{1}{\ve}(-U'(x_t)+U'(2a(1+t)-x_t))\, \dd t + \frac{\sigma}{\sqrt{\ve}} \dd W_t \nonumber
\end{align}
This equation is well-defined as long as $2a(1+t) - b  < x_t < b$, which is the same condition that ensures the distance between any neighbouring particles is less than $b$.  As soon as this inequality fails, we consider the chain to be broken as there is no longer any interaction between $x$ and one of its neighbours.  Let
\be\label{tau0}
\tau = \inf\{t \geq 0 : x_t \notin (2a(1+t)-b,b)\}
\ee
We say the chain breaks on the left-hand side if $x_{\tau} = b$ and it breaks on the right-hand side if $x_{\tau} = 2a(1+\tau)-b$.  The chain necessarily breaks when $t = b/a-1$, so $\tau \leq b/a-1$.

Let $\Prob$ denote the law of the process $x_t$ when started from $a$ at time $0$.  We also write $f(\sigma) \ll g(\sigma)$ to mean that $f(\sigma)/g(\sigma) \to 0$ as $\sez$.
\begin{theorem}\label{main}
Let $x_t$ solve (\ref{eq:xt}) and define $\tau$ as in (\ref{tau0}).
\begin{enumerate}
 \item (Fast stretching) If $\sigma |\ln \sigma|^{1/2} \ll \ve(\sigma) \ll 1$ then $\Prob\{x_{\tau} = b\} \to 0$ as $\sez$.
 \item (Slow stretching) If
\[
 \frac{1}{\si^{2/3}}\exp\left\{-\frac{1}{\si^{2/3}}\right\} \ll \ve(\sigma) \ll \sigma|\ln \sigma|^{-1/2}
\]
then $\Prob\{x_{\tau} = b\} \to 1/2$ as $\sez$.
\end{enumerate}
\end{theorem}
 \begin{itemize}
  \item The proof of this theorem will actually yield that when $\sigma |\ln \sigma|^{1/2} \ll \ve(\sigma) \ll 1$, $\Prob\{x_{\tau} = b\} < (C\ve/\si^2)\e{-c \ve^2/\si^2}$.
 \item The lower bound on $\ve$ in (2) arises because our method applies on timescales shorter than Kramers' time, although we expect the result to hold without this lower bound.
 \item If $U$ is quadratic, then (2) is true without the lower bound on $\ve$.  We will comment on this at the end.
 \item The proof can be extended to the case that the chain is stretched according to some non-linear function $p(t)$, that is, $x_R(t) = 2a(1+p(t))$ where $0 < p_0 < p'(t) < p_1$.
 \end{itemize}
The theorem shows that when the stretching is fast, the chain will almost surely break on the right-hand side as $\sez$.  This is the same behaviour as in the deterministic case when $\si = 0$ (see the following section).  However, when the stretching is sufficiently slow, there is an equal probability to break on either side, as when there is no stretching at all.

\section{Proof}
\subsection{An alternative formulation}
For times $t < \tau$, we can replace $U$ with any potential $\tU \in \caC(\R)$ such that $\tU \in \caC^3((0,\infty))$, $\tU(y) = U(y)$ for $|y| \leq b$, $\tU(-y) = \tU(y)$ and $\tU''(y) \geq u_0 > 0$ for all $|y| > a_0$.  Defining $\tH(x,t) = \tU(x) + \tU(2a(1+t)-x)$ we have that for times $t < \tau$, $H(x_t,t) = \tH(x_t,t)$ and $x_t$ also solves
\[
\dd x_t = -\frac{1}{\ve}\frac{\partial \tH}{\partial x}(x_t,t) \, \dd t + \frac{\sigma}{\sqrt{\ve}} \dd W_t
\]
Let $x_t^{\deter}$ be the solution of the deterministic equation
\be\label{eq:ode}
 \dd x_t^{\deter} = -\frac{1}{\ve}\frac{\partial \tH}{\partial x}(x_t^{\deter},t)\dd t
\ee
with $x_0^{\deter} = a$.  This ODE is well-defined as long as $0 < x_t^{\deter} < 2a(1+t)$.  Since its solution can be written
\be\label{eq:det}
 x_t^{\deter} = a(1+t)-\frac{\ve}{\tU''(a(1+t))}+\mcO(\ve^2)
\ee
we see that for small $\ve$ this condition holds, in particular, for all $t \in [0,b/a-1]$.  Furthermore, by taking $\ve$ sufficiently small we also have in this interval that $a_0 < x_t^{\deter} < 2a(1+t) - a_0$.  If we had used $H$ instead of $\tH$ in (\ref{eq:ode}), then $x_t^{\deter}$ would not have been defined on the whole interval $[0,b/a-1]$.  Indeed, there is $t < b/a-1$ such that $2a(1+t)-x_t^{\deter}=b$.

We can now define the deviation process $y_t \isdefby x_t-x_t^{\deter}$ on the interval $[0,(b/a-1 )\wedge \tau]$.   This solves, with initial condition $y_0 = 0$,
\begin{align}
 \dd y_t & = \frac{1}{\ve}[-\tU'(x_t)+\tU'(x_t^{\deter})+\tU'(2a(1+t)-x_t)-\tU'(2a(1+t)-x_t^{\deter})]\dd t+\frac{\sigma}{\sqrt{\ve}}\dd W_t \nonumber \\ 
& = \frac{1}{\ve}[A(t)y_t+B(y_t,t)]\dd t+ \frac{\sigma}{\sqrt{\ve}}\dd W_t \label{eq:y}
\end{align}
where
\[
 A(t) = -\tU''(x_t^{\deter})-\tU''(2a(1+t)-x_t^{\deter})
\]
and there is a constant $M>0$ such that $|B(y,t)| \leq My^2$ for all pairs $(y,t) \in \caD$, where $\caD$ is given in (\ref{eq:D}).  We can also find constants $A_0,A_1 > 0$ such that $-A_1 \leq A(t) \leq -A_0$ for all $t \in [0,b/a-1]$.

For the chain to be unbroken, $y_t$ must satisfy
\[
 2a(1+t)-b-x_t^{\deter} < y_t < b-x_t^{\deter}
\]
which we write as
\[
 d_{-}(t)<y_t<d_{+}(t)
\]
where
\[
 d_{+}(t) = b-a(1+t)+\frac{\ve}{\tU''(a(1+t))}+\mcO(\ve^2)
\]
and
\[
 d_{-}(t) = a(1+t)-b+\frac{\ve}{\tU''(a(1+t))}+\mcO(\ve^2)
\]
The problem is then to study the first exit of the process $y_t$ from the space-time domain, $\mcD = \mcD(\ve)$, given by
\be\label{eq:D}
 \mcD = \{(y,t):d_{-}(t)<y<d_{+}(t), 0 \leq t \leq b/a-1\}
\ee
The stopping time $\tau$ given in (\ref{tau0}) can be written
\be\label{eq:tau}
\tau = \inf\{t \geq 0 : y_t \notin \caD\}
\ee
Then $y_{\tau} = d_-(\tau)$ corresponds to $x_{\tau} = 2a(1+\tau)-b$, that is, the chain breaking on the right-hand side.
Note that $d_+(t) \geq -d_-(t)$ for all $t \in [0,b/a-1]$ and so the curve $d_{-}(t)$ crosses zero before $d_{+}(t)$.  This means that in the deterministic case, when $y_t \equiv 0$, the curve $d_{-}(t)$ is hit before $d_{+}(t)$ and the chain breaks on the right-hand side.

If we let $\Prob^{t_0,y_0}$ denote the law of the process $y_t$ when started from $y_0$ at time $t_0$, then Theorem \ref{main} can be stated as
\begin{theorem}[Alternative version of Theorem \ref{main}]Let $y_t$ solve (\ref{eq:y}) and define $\tau$ as in (\ref{eq:tau}).
\begin{enumerate}
 \item If $\sigma |\ln \sigma|^{1/2} \ll \ve(\sigma) \ll 1$ then $\Prob^{0,0}\{y_{\tau} = d_+(\tau)\} \to 0$ as $\sez$.
 \item If
\[
 \frac{1}{\si^{2/3}}\exp\left\{-\frac{1}{\si^{2/3}}\right\} \ll \ve(\sigma) \ll \sigma|\ln \sigma|^{-1/2}
\]
then $\Prob^{0,0}\{y_{\tau} = d_+(\tau)\} \to 1/2$ as $\sez$.
\end{enumerate}

\end{theorem}
The main idea when proving this theorem is as follows.  The process $y_t$ is given by
\begin{align*}
 y_t & = \frac{\sigma}{\sqrt{\ve}}\int_0^t \e{\alpha(t,s)/\ve} \dd W_s + \frac{1}{\ve}\int_0^t \e{\alpha(t,s)/\ve}B(y_s,s)\dd s\\
& =: y_t^0 + y_t^1
\end{align*}
where $\alpha(t,s) = \int_s^t A(u)\dd u$ satisfies $-A_1(t-s) \leq \alpha(t,s) \leq -A_0(t-s)$.  We will also write $\alpha(t) = \alpha(t,0)$.  The term $y_t^0$ is Gaussian and so is easier to work with than $y_t^1$.  As long as $y_t$ is not too large, then $y_t^1$ can be bounded using that $|B(y,t)| \leq My^2$.  For example, if $t<\tau$ and $\sup_{0 \leq s \leq t} |y_s|\leq D$, then
\be\label{y13}
|y_t^1| \leq \frac{1}{\ve}\int_0^t |B(y_s,s)| \e{\alpha(t,s)/\ve} \dd s \leq \frac{MD^2}{\ve}\int_0^t \e{-A_0(t-s)/\ve} \dd s = \frac{MD^2}{A_0}(1-\e{-A_0 t/\ve})
\ee
If $D$ is small, then the contribution of $y_t^1$ will be much less than that of $y_t^0$.  The following proposition tells us for which $\ve$ we have this type of bound.
\begin{proposition}\label{prop:yleqD}
Let $\si \ll D(\si) \ll 1$ be such that
\[
 \frac{D^2}{\si^2}\exp\left\{-\frac{D^2}{\si^2}\right\} \ll \ve(\si) \ll 1
\]
Then
\[
 \lim_{\sez}\Prob^{0,0}\left\{\sup_{0 \leq t \leq (b/a-1 )\wedge \tau}|y_t| \geq D\right\} = 0 
\]
\end{proposition}
\begin{remark}
 The lower bound on $\ve$ is related to the fact that we cannot bound $y_t$ on timescales larger than Kramers' time.  An excursion of size $D$ corresponds to climbing a potential height of $\mcO(D^2)$, which we expect to occur after a time of order $\e{D^2/\si^2}$.
\end{remark}
To prove this proposition, we will use a lemma which says roughly that the Gaussian term $y_t^0$ stays with high probability in a corridor of width proportional to its variance.  More precisely, the variance of $y_t^0$ is given by
\[
 \Var(y_t^0) = \frac{\sigma^2}{\ve}\int_0^t \e{2\alpha(t,s)/\ve}\dd s = \sigma^2 v(t)
\]
where $v(t)$ is a solution of $\ve \dot{v} = 2A(t)v+1$ with $v(0)=0$.  Following \cite{nb}, we see that since the right-hand side of this ODE vanishes for $v = -1/(2A(t))$, we can find a particular solution of the form
\be\label{eq:xi}
 \xi(t) = -\frac{1}{2A(t)}+\mcO(\ve)
\ee
where the $\mcO(\ve)$ term is uniform in $t$.  So there are constants $0 < \xi_- < \xi_+$ such that for $\ve$ sufficiently small, $\xi_- \leq \xi(t) \leq \xi_+$ for all $t \in [0,b/a-1]$.  The function $\xi(t)$ satisfies $|\xi(t)-v(t)| \leq C\e{2\alpha(t)/\ve}$ and will be used in the following lemma.  The advantage of $\xi(t)$ over $v(t)$ is that it is bounded away from zero.  The following lemma will be applied in cases where $\si \ll H$ and shows how paths of $y_t^0$ are concentrated.
\begin{lemma}[Berglund, Gentz \cite{nb,nb:pathwise}]\label{bglem2}
If $H^2 > 2\sigma^2$ then for any $t \in [0,b/a-1]$, we have
\[
 \Prob^{0,0}\left\{\sup_{0 \leq s \leq t} \frac{|y_s^0|}{\sqrt{\xi(s)}} \geq H\right\} = C_{H/\sigma}(t,\ve)\e{-H^2/2\sigma^2}
\]
with
\[
 C_{H/\sigma}(t,\ve) \leq 2\e{}\left\lceil \frac{|\alpha(t)|}{\ve}\frac{H^2}{\sigma^2}[1+\mcO(\ve)]\right\rceil
\]
\end{lemma}
This lemma is proved by partitioning the interval $[0,t]$ and applying on each subinterval the inequality
\[
\Prob\left\{\sup_{0 \leq s \leq t} \left|\int_0^s \varphi(u)\dd W_u\right| \geq \delta \right\} \leq 2\exp\left\{-\frac{\delta^2}{2\int_0^t \varphi(u)^2 \dd u}\right\}
\]
which is valid for deterministic Borel-measurable functions $\varphi: [0,t] \to \R$.
\begin{proof}[Proof of Proposition \ref{prop:yleqD}]
Consider the stopping time given by
\[
 \tau(h) = \inf\{t \geq 0: |y_t| \geq h\sqrt{\xi(t)}\}
\]
where $\xi(t)$ is given in (\ref{eq:xi}).  If we can find $h$ such that $h \sqrt{\xi(t)} \leq D$ for all $t \in [0,b/a-1]$ then
\begin{align*}
 \Prob^{0,0}\left\{\sup_{0 \leq t \leq (b/a-1 )\wedge \tau}|y_t| \geq D\right\} & \leq \Prob^{0,0}\left\{\tau(h)<(b/a-1 )\wedge \tau \right\}\\
 & = \Prob^{0,0}\left\{\sup_{0 \leq t \leq (b/a-1 )\wedge \tau \wedge \tau(h)} \frac{|y_t|}{\sqrt{\xi(t)}} \geq h\right\}
\end{align*}
As was noted above, $\xi(t) \leq \xi_+$ and so choosing $h = D/\sqrt{\xi_+}$ gives $h \sqrt{\xi(t)} \leq D$.  For $0 \leq t \leq (b/a-1 )\wedge \tau \wedge \tau(h)$, we have $|y_t| \leq D$ and so $|y_t^1| \leq M D^2 /A_0$.  Therefore,
\[
\sup_{0 \leq t \leq (b/a-1 )\wedge \tau \wedge \tau(h)} \frac{|y_t|}{\sqrt{\xi(t)}} \leq \sup_{0 \leq t \leq b/a-1} \frac{|y_t^0|}{\sqrt{\xi(t)}} + \frac{M D^2}{A_0 \sqrt{\xi_-}}
\]
and we have
\begin{multline*}
\Prob^{0,0}\left\{\sup_{0 \leq t \leq (b/a-1 )\wedge \tau \wedge \tau(h)} \frac{|y_t|}{\sqrt{\xi(t)}} \geq \frac{D}{\sqrt{\xi_+}}\right\} \\ \leq \Prob^{0,0}\left\{\sup_{0 \leq t \leq b/a-1} \frac{|y_t^0|}{\sqrt{\xi(t)}} \geq D\left(\frac{1}{\sqrt{\xi_+}} - \frac{M D}{A_0 \sqrt{\xi_-}}\right)\right\}
\end{multline*}
We can apply Lemma \ref{bglem2} with $H = D(1/\sqrt{\xi_+} - MD /(A_0 \sqrt{\xi_-})) = D(1/\sqrt{\xi_+}+\mcO(D))$.
\begin{multline*}
\Prob^{0,0}\left\{\sup_{0 \leq t \leq b/a-1} \frac{|y_t^0|}{\sqrt{\xi(t)}} \geq D\left(\frac{1}{\sqrt{\xi_+}} - \frac{M D}{A_0 \sqrt{\xi_-}}\right)\right\} \\  \leq 2\e{}\left\lceil C_1 \frac{D^2}{\ve \si^2}(1+ \mcO(D+\ve))\right\rceil \exp\left\{-C_2 \frac{D^2}{\si^2}(1+\mcO(D))\right\}
\end{multline*}
for constants $C_1,C_2 > 0$, from which the result follows.
\end{proof}

\subsection{Fast stretching}
In this case, we show that the chain is stretched so fast that the process $y_t$ is almost surely never greater than $d_{+}(b/a-1)$ in absolute value.  Note that the curve $d_{+}(t)$ is decreasing, since its derivative is given by 
\[
-\dd x^{\deter}_t/\dd t = \tU'(x_t^{\deter})-\tU'(2a(1+t)-x_t^{\deter}) < 0
\]
using that $\tU''(y) \geq u_0 > 0$ for $|y|>a_0$ and $a_0 < x_t^{\deter} < 2a(1+t) - x_t^{\deter}$ for $t>0$, which can be seen from (\ref{eq:det}).  Since the curve $d_{+}(t)$ is decreasing, this means that it cannot have ever been hit by the process $y_t$ and so the chain must have broken on the right-hand side (Fig \ref{fig:corr}).  This is contained in the following proposition.
\begin{figure}
    \centering
    \includegraphics[width=.8\textwidth]{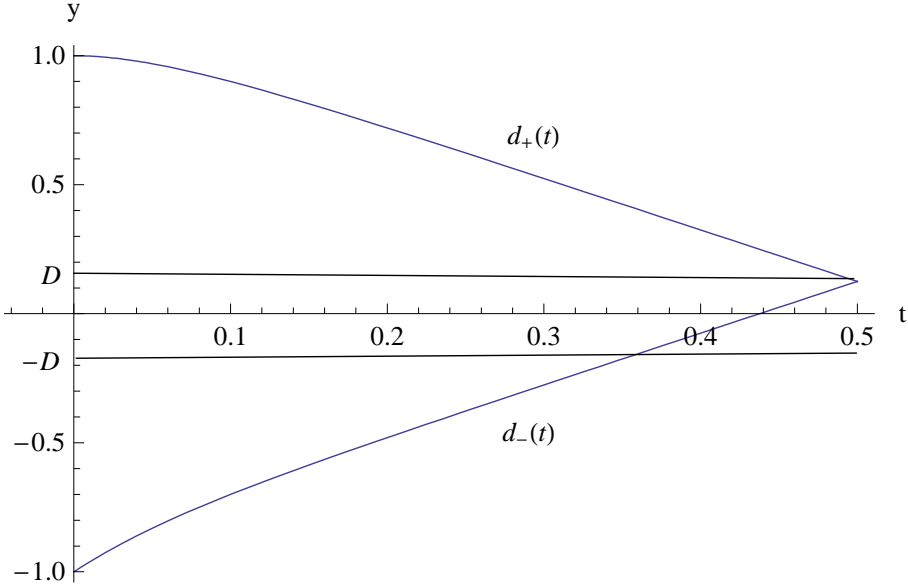}
\caption{When the chain is pulled sufficiently fast, the process $y_t$ is unlikely to leave the corridor of width $D = d_{+}(b/a-1)$, so must hit $d_{-}(t)$ first (shown here for $U(y) = y^2-4y+3$ and $\ve = 0.25$).}
\label{fig:corr}
\end{figure}
\begin{proposition}\label{prop:fast}
Let $\sigma |\ln \sigma|^{1/2} \ll \ve(\si) \ll 1$.  Then
\[
 \lim_{\sez}\Prob^{0,0}\left\{\sup_{0 \leq t \leq (b/a-1 )\wedge \tau}|y_t| \geq d_{+}(b/a-1)\right\} = 0 
\]
\end{proposition}
\begin{proof}
Apply Proposition \ref{prop:yleqD} with $D = d_+(b/a-1) = \ve/\tU''(b/a-1)+\mcO(\ve^2)$.
\end{proof}
\subsection{Slow stretching}
The strategy is as follows.  Suppose we are given $D$ such that
\[
\lim_{\sez}\Prob^{0,0}\left\{\sup_{0 \leq t \leq (b/a-1) \wedge \tau} |y_t|\geq D\right\} = 0
\]
Then we can assume that $|y_t^1| \leq MD^2/A_0$ for all $t<\tau$, since all other cases have zero probability in the limit.  To simplify notation, we will write this last inequality as $|y_t^1| \leq D^2$.  For all $t < \tau$ we have
\[
 y_t^0 - D^2 \leq y_t \leq y_t^0 + D^2
\]
and, therefore,
\[
 P_L \leq \Prob^{0,0}\{y_{\tau} = d_+(\tau)\} \leq P_U
\]
where
\[
 P_L = P_L(D) = \Prob^{0,0}\{y_t^0 - D^2 \text{ hits $d_+(t)$ before $d_-(t)$}\}
\]
and
\[
 P_U = P_U(D) = \Prob^{0,0}\{y_t^0 + D^2 \text{ hits $d_+(t)$ before $d_-(t)$}\}
\]
The aim of this section is to show that given $\ve(\si)$, we can pick $D(\si)$ such that $P_L$ and $P_U$ tend to $1/2$, which gives the result.  The proof of each limit is similar, so we will show the details for $P_L$ only.  Note that $P_L$ can be written 
\[
 P_L = \Prob^{0,0}\{y_t^0 \text{ hits $d_+(t) + D^2$ before $d_-(t)+D^2$}\}
\]
Define the stopping time
\[
 \tau_L = \tau_L(D) = \inf\{t \geq 0 : |y_t^0| \geq -d_-(t)-D^2\} < T
\]
where $T = \inf\{t \geq 0: -d_-(t)-D^2 = 0\}$.  By symmetry,
\[
 \Prob^{0,0}\{y_{\tau_L}^0 = -d_-(t)-D^2\} = \Prob^{0,0}\{y_{\tau_L}^0 = d_-(t)+D^2\} = \frac{1}{2}
\]
We must show that if $y_{\tau_L}^0 = -d_-(\tau_L)-D^2$ then almost surely $y_t^0$ hits $d_+(t)+D^2$ soon after as $\sez$ (see Fig \ref{boundary}).

\begin{figure}
  \centering
    \includegraphics[width=.8\textwidth]{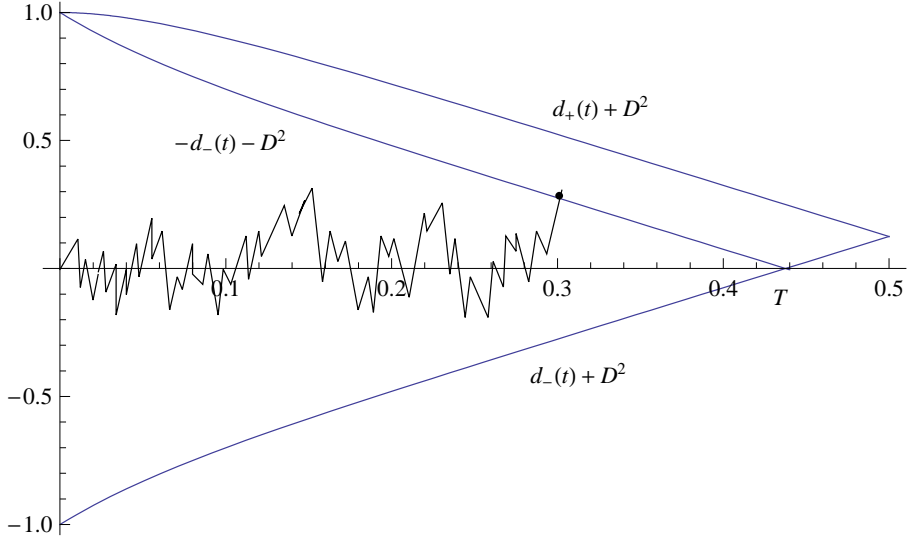}
  \caption{When the chain is stretched slowly, we show that conditional probability of hitting the curve $d_+(t)+D^2$ before $d_{-}(t)+D^2$, when started from $-d_{-}(t)-D^2$, goes to one as $\sez$.}
\label{boundary}
\end{figure}
In the following two lemmas, we establish upper and lower bounds for $\tau_L$.  The upper bound is needed in order that $y_t^0$, when started from $-d_-(t)-D^2$, is much closer to $d_+(t)+D^2$ than to $d_-(t)+D^2$.  If $\tau_L$ is too close to $T$ then this is not the case and $y_t^0$ is more likely to exit in the ``wrong direction''.  The lower bound is required since we cannot expect the conditional probability of hitting $d_+(t)+D^2$, when started from $-d_-(t)-D^2$, to be close to one if it is unlikely that $y_t^0$ has even reached $-d_-(t)-D^2$.

We will now write $\Prob^{t_0,y_0}$ to denote the law of $y_t^0$ when started from $y_0$ at time $t_0$.
\begin{lemma}\label{lem:up}
For any $f(\si) \ll 1$ and $D^2 \ll \si$,
\[
\lim_{\sez} \Prob^{0,0}\left\{\tau_L \leq \frac{b}{a} - 1 -\frac{\sigma f(\sigma)}{a}\right\} = 1
\]
\end{lemma}
\begin{proof}
To prove this upper bound for $\tau_L$, we use a simple fixed-time estimate.  Let $t = b/a-1 - \sigma f(\sigma)/a$.  If $f$ is such that $t \geq T$, then the upper bound is trivial.  Otherwise
\begin{align*}
\Prob^{0,0}\{\tau_L \leq t\} & \geq \Prob^{0,0}\left\{|y_t^0|\geq -d_-(t)-D^2\right\}\\
& = \frac{2}{\sqrt{2\pi \Var(y_t^0)}}\int_{-d_-(t)-D^2}^{\infty}\exp\left\{-\frac{x^2}{2\Var(y_t^0)}\right\}\dd x\\
& = \frac{2}{\sqrt{\pi}}\int_{(-d_-(t)-D^2)/\sqrt{2\Var(y_t^0)}}^{\infty}e^{-z^2}\dd z
\end{align*}
where 
\[
\frac{-d_-(t)-D^2}{\sqrt{2\Var(y_t^0)}} \leq \frac{ f(\si)-D^2/\si + \mcO(\ve/\si)}{\sqrt{1/A_1+\mcO(\ve^n)}}, \text{ $n \geq 1$ }
\]
This last inequality follows since $\Var(y_t^0) \geq \si^2 (1-\e{2A_1 t/\ve})/A_1$.  The right-hand side goes to zero as $\sez$.
\end{proof}
\begin{lemma}\label{lem:low}
For any $f$ such that  $1 \ll f(\si) \ll 1/\si$ and
\begin{equation}\label{f2}
 f(\si)^2 \exp\{-f(\si)^2\} \ll \ve \ll 1
\end{equation}
and for any $D^2 \ll \si$, we have
\[
 \lim_{\sez}\Prob^{0,0}\left\{\tau_L \geq \frac{b}{a}-1-\frac{\sigma f(\sigma)}{a}\right\} = 1
\]
\end{lemma}
\begin{proof}
We will use Lemma \ref{bglem2}.  First note that there is $c_1>0$ such that for $\ve(\si)$ sufficiently small,
\[
 -d_-(t) \geq b-a(1+t)-c_1 \ve
\]
holds for all $t$.  Putting $t_1 = b/a-1-\si f(\si)/a$, we get
\[
 \inf_{0 \leq t \leq t_1}(-d_-(t)) \geq \si f(\si)-c_1 \ve
\]
Now put $H= (\si f(\si)-c_1 \ve -D^2)/\sqrt{\xi_+}>0$.  Then
\[
 H\sqrt{\xi(t)} \leq H\sqrt{\xi_+} = \si f(\si)-c_1 \ve -D^2 \leq \inf_{0 \leq t \leq t_1}(-d_-(t)-D^2)
\]
Therefore,
\begin{align*}
 \Prob^{0,0}\left\{\tau_L < t_1 \right\} & \leq \Prob^{0,0}\left\{\sup_{0 \leq t \leq t_1} \frac{|y_t^0|}{\sqrt{\xi(t)}} \geq \inf_{0 \leq t \leq t_1} \frac{-d_-(t)-D^2}{\sqrt{\xi(t)}}\right\}\\
& \leq \Prob^{0,0}\left\{\sup_{0 \leq t \leq t_1} \frac{|y_t^0|}{\sqrt{\xi(t)}} \geq H\right\}
\end{align*}
Now we apply Lemma \ref{bglem2} to show that the right-hand side of this inequality tends to $0$ as $\sez$, which gives the result.
\end{proof}
Suppose that $f_+$ satisfies $1 \ll f_+(\si) \ll \min(\si/D^2,\si/\ve)$.  Then we can apply Lemmas \ref{lem:up} and \ref{lem:low} to $1/f_+(\si)$ and $f_+(\si)$, respectively.  This tells us that
\be\label{eq:t^*}
 \lim_{\sez}\Prob^{0,0}\left\{\frac{b}{a}-1-\frac{\si f_+(\si)}{a} \leq \tau_L \leq \frac{b}{a}-1-\frac{\si }{a f_+(\si)}\right\} = 1
\ee
The next proposition has three parts.  Together, they show that if $y_t^0$ starts from $-d_-(t)-D^2$ for suitable times $t^*$ as given in (\ref{eq:t^*}), then it hits $d_+(t)+D^2$ in a small interval $[t^*,t^*+\Delta]$ afterwards and does not hit the lower curve $d_-(t)+D^2$ in this time.

Recall that $T = \inf\{t \geq 0: -d_-(t)-D^2=0\}$.
\begin{proposition}\label{prop:3}
Let $f_+(\si)$ and $\Delta(\si)$ be chosen so that
\be\label{eq:f+}
 1 \ll f_+(\si) \ll \sqrt{\ve/\Delta} \ll \min(\si/D^2,\si/\ve)
\ee
Then for every $f$ such that $1/f_+(\si) \leq f(\sigma) \leq f_+(\si)$ and $t^* \isdefby b/a-1-\sigma f(\sigma)/a$, we have
\begin{enumerate}
 \item $[t^*,t^* + \Delta] \subset [0,T]$ for $\si$ sufficiently small.\\
 \item $\displaystyle \lim_{\sez}\Prob^{t^*,-d_-(t^*)-D^2}\left\{y_t^0 \geq d_+(t) + D^2 \text{ for some $t \in [t^*,t^*+\Delta]$}\right\} = 1$
 \item $\displaystyle \lim_{\sez}\Prob^{t^*,-d_-(t^*)-D^2}\left\{\inf_{t^* \leq t \leq t^* + \Delta} y_t^0 < 0 \right\} = 0$
\end{enumerate}
\end{proposition}
\begin{remark}
 Note that (1) and (3) together guarantee that $y_t^0$ does not hit $d_-(t)+D^2$ in the interval $[t^*,t^*+\Delta]$.
\end{remark}
\begin{proof}
(1) The upper bound on $f_+$ implies, in particular, that $f_+ \ll 1/\si$ and so $t^* > 0$.  As in the proof of Lemma \ref{lem:low}, for sufficiently small $\ve$ we have the uniform bound
\[
 -d_-(t)-D^2 \geq b-a(1+t) -c_1 \ve -D^2
\]
If $t< b/a - 1 - (c_1\ve +D^2)/a$ then the right-hand side is positive and $t<T$.  Note that
\begin{align*}
 t^*+\Delta & = \frac{b}{a}-1 -\frac{\si f(\si)}{a} + \Delta\\
& = \frac{b}{a}-1 -\frac{\si f(\si) - a\Delta}{a}
\end{align*}
If $\si f(\si) - a\Delta > c_1 \ve+D^2$ for sufficiently small $\si$ then $t^*+\Delta < T$.  By (\ref{eq:f+}), this is indeed satisfied.
\\[2mm]
(2) We show that $y_t^0$ hits $d_+(t)+D^2$ in the interval $[t^*,t^*+\Delta]$, which is the same as the process $\e{-\alpha(t)/\ve}y_t^0$ hitting the curve $\e{-\alpha(t)/\ve}(d_+(t)+D^2)$.  The latter will be more convenient to show, since it will lead to a probability involving a Gaussian martingale, for which the reflection principle can be applied.  The process $y_t^0$, when started from $-d_-(t^*)-D^2$ at time $t^*$, is given by
\be\label{eq:ycond}
y_t^0 = -\e{\alpha(t,t^*)/\ve}(d_-(t^*)+D^2) + \frac{\sigma}{\sqrt{\ve}}\int_{t^*}^t \e{\alpha(t,s)/\ve}\dd W_s
\ee
from which we deduce that
\begin{align}
 \e{-\alpha(t)/\ve}y_t^0 & = -\e{-\alpha(t^*)/\ve}(d_-(t^*)+D^2) + \frac{\sigma}{\sqrt{\ve}}\int_{t^*}^t \e{-\alpha(s)/\ve}\dd W_s \nonumber\\ 
 & \bydefis -\e{-\alpha(t^*)/\ve}(d_-(t^*)+D^2) + z_t^0 \label{z0}
\end{align}
For all $t \in [t^*,t^*+\Delta]$, we have
\[
\e{-\alpha(t)/\ve}(d_+(t)+D^2) \leq \e{-\alpha(t^*+\Delta)/\ve}(d_+(t^*)+D^2)
\]
Define $h(t^*,\Delta)\isdefby \e{-\alpha(t^*+\Delta)/\ve}(d_+(t^*)+D^2) +\e{-\alpha(t^*)/\ve}(d_-(t^*)+D^2) > 0$.  Then if
\[
 \sup_{t^* \leq t \leq t^*+\Delta}z_t^0 \geq h(t^*,\Delta)
\]
we must have that $z_t^0 \geq \e{-\alpha(t)/\ve}(d_+(t)+D^2) + \e{-\alpha(t^*)/\ve}(d_-(t^*)+D^2)$ for some $t \in [t^*,t^*+\Delta]$, which is equivalent to $y_t^0 \geq d_+(t)+D^2$.  By the reflection principle applied to $z_t^0$, we have
\begin{align*}
\Prob^{t^*,-d_-(t^*)-D^2}\left\{\sup_{t^* \leq t \leq t^*+\Delta}z_t^0 \geq h(t^*,\Delta)\right\} & = 2\Prob^{t^*,-d_-(t^*)-D^2}\left\{z^0_{t^*+\Delta} \geq h(t^*,\Delta)\right\}\\
& = \frac{2}{\sqrt{\pi}}\int_L^{-\infty}\e{-z^2}\dd z
\end{align*}
where $L = h(t^*,\Delta)/\sqrt{2\Var(z^0_{t^*+\Delta})}$.  If we can show that $L \to 0$ as $\sez$, then we will be done.  We have
\be\label{eq:firstpart}
0 \leq \frac{h(t^*,\Delta)}{\sqrt{2\Var(z^0_{t^*+\Delta})}} = \frac{\e{-\alpha(t^*+\Delta,t^*)/\ve}(d_+(t^*)+D^2)+d_-(t^*)+D^2}{\sqrt{2 \e{2\alpha(t^*)/\ve}\Var(z^0_{t^*+\Delta})}}
\ee
where
\[
 2 \e{2\alpha(t^*)/\ve}\Var(z^0_{t^*+\Delta}) = \frac{2\sigma^2}{\ve}\int_{t^*}^{t^*+\Delta}\e{-2\alpha(s,t^*)/\ve}\dd s \geq \frac{\sigma^2}{A_0}(\e{2A_0 \Delta/\ve}-1)
\]
which means that
\[
 \frac{\e{-\alpha(t^*+\Delta,t^*)/\ve}(d_+(t^*)+D^2)+d_-(t^*)+D^2}{\sqrt{2 \e{2\alpha(t^*)/\ve}\Var(z^0_{t^*+\Delta})}} \leq \frac{\e{A_1 \Delta/\ve}(d_+(t^*)+D^2)+d_-(t^*)+D^2}{\si \sqrt{(\e{2A_0 \Delta/\ve}-1)/A_0}}
\]
Using Taylor expansions for the exponential terms, we see from (\ref{eq:f+}) that the right-hand side tends to zero as $\sez$.
\\[2mm]
(3) Since the distribution of $y_t^0$, when started at $0$, is symmetric about $y=0$, $y_t^0$ satisfies a reflection principle about this line (see Appendix of \cite{nb}) and we have
\be\label{eq:zero}
 \Prob^{t^*,-d_-(t^*)-D^2}\left\{\inf_{t^* \leq t \leq t^* + \Delta} y_t^0 < 0\right\} = 2 \Prob^{t^*,-d_-(t^*)-D^2}\left\{y_{t^*+\Delta}^0 < 0\right\}
\ee
where we recall from (\ref{eq:ycond}) that the conditional process is given by
\[
y_t^0 = -\e{\alpha(t,t^*)/\ve}(d_-(t^*)+D^2) + \frac{\sigma}{\sqrt{\ve}}\int_{t^*}^t \e{\alpha(t,s)/\ve}\dd W_s
\]
Therefore,
\[
2 \Prob^{t^*,-d_-(t^*)-D^2}\left\{y_{t^*+\Delta}^0 \leq 0\right\} = \frac{2}{\sqrt{\pi}}\int_{-\infty}^U \e{-z^2}\dd z
\]
where
\begin{align*}
U = \frac{\e{\alpha(t^*+\Delta,t^*)/\ve}(d_-(t^*)+D^2)}{\sqrt{2\Var(y_{t^*+\Delta}^0)}} \leq \frac{d_-(t^*)-D^2}{\si \sqrt{(\e{2A_1 \Delta/\ve}-1)/A_1}}
\end{align*}
This inequality for $U$, which is negative, comes from the bound
\[
 2\e{-2\alpha(t^*+\Delta,t^*)/\ve} \Var(y_{t^*+\Delta}^0) \leq \si^2 (\e{2A_1 \Delta/\ve}-1)/A_1
\]
Again using Taylor expansions and (\ref{eq:f+}), we see that the upper bound for $U$ goes to $-\infty$ as $\sez$.
\end{proof}
\begin{proof}[Proof of Theorem \ref{main}]
First we suppose that $\si^2 |\ln \si| \ll \ve \ll \si |\ln \si|^{-1/2}$.  Then we can pick $D$ such that $\si^2 |\ln \si| \ll D^2 \ll \ve$ and it follows that
\[
 \frac{D^2}{\si^2}\exp\left\{-\frac{D^2}{\si^2}\right\} \ll \si^2 \ll \ve
\]
This means we can apply Proposition \ref{prop:yleqD} to show $|y_t|$ remains bounded by $D$ almost surely as $\sez$.  We can choose $|\ln \si|^{1/2} \ll f_+(\si) \ll \si/\ve = \min(\si/D^2,\si/\ve)$ so that
\[
 f_+(\si)^2\exp\left\{-f_+(\si)^2\right\} \ll \si^2 \ll \ve
\]
and can apply Lemmas \ref{lem:up} and \ref{lem:low} to $1/f_+$ and $f_+$, respectively, to show that we only need to consider hitting times of $-d_-(t)-D^2$ of the form $t^* = b/a-1-\si f(\si)/a$, where $1/f_+(\si) \leq f(\si) \leq f_+(\si)$.  Then we can apply Proposition \ref{prop:3} to show that the conditional probability of hitting $d_+(t)+D^2$ before $d_-(t)+D^2$ goes to one.

Now suppose that
\[
 \frac{1}{\si^{2/3}}\exp\left\{-\frac{1}{\si^{2/3}}\right\} \ll \ve \ll \si^2 |\ln \si|
\]
Pick $D$ such that $\si^2 |\ln \si| \ll D^2 \ll \si^{4/3}$ and
\[
 \frac{1}{\si^{2/3}}\exp\left\{-\frac{1}{\si^{2/3}}\right\} \ll \frac{D^2}{\si^2}\exp\left\{-\frac{D^2}{\si^2}\right\} \ll \ve
\]
We can again apply Proposition \ref{prop:yleqD} to bound $|y_t|$ by $D$.  Letting $f_+(\si) = D/\si \ll \si/D^2 = \min(\si/D^2,\si/\ve)$, we can apply Lemmas \ref{lem:up} and \ref{lem:low} to $1/f_+$ and $f_+$, respectively.  Then Proposition \ref{prop:3} holds and we are done.

Note that taking instead $\si^{4/3} \ll D^2 \ll \si$ gives the same lower bound on $\ve$ because in that case we need $\ve \gg f_+(\si)^2 \exp\{-f_+(\si)^2\}$ where $f_+(\si) \ll \si/D^2 \ll \si^{-1/3}$.
\end{proof}
We end this paper by commenting on the case of a quadratic potential $U$.  For such potentials, there is no non-linear term, $y_t^1$, and $D \equiv 0$.  We just have to show that $y_t^0$ has probability $1/2$ to hit $d_+(t)$ before $d_-(t)$.  For this we consider the conditional probability to hit $d_+(t)$ when starting from $-d_-(t)$.  If we define the analogue of $\tau_L$ as $\tilde{\tau}_L = \inf\{t \geq 0: |y_t^0| \geq -d_-(t)\}$ then when $\ve \ll \si^2$ we can show this conditional probability goes to one with only an upper bound for $\tilde{\tau}_L$.  Since we do not need to bound $|y_t^0|$, no lower bound on $\ve$ is required.  For $\si^2 \ll \ve \ll \si |\ln \si|^{-1/2}$, a lower bound on $\tilde{\tau}_L$ is needed to show the conditional probability goes to one, but this holds for such $\ve$ without additional assumptions.

\end{document}